\documentclass[12pt]{amsart}
\usepackage{amsfonts,latexsym,amsmath, amssymb}
\usepackage{mathrsfs,MnSymbol}
\usepackage{url,color}
\usepackage{upgreek}
\usepackage{fancyhdr}
\usepackage{hyperref}
\usepackage{times}
\usepackage{scalefnt}
\usepackage{url}
\usepackage{cite}

\newcommand{\bea}{\begin{eqnarray}}
\newcommand{\eea}{\end{eqnarray}}
\def\beaa{\begin{eqnarray*}}
\def\eeaa{\end{eqnarray*}}
\def\ba{\begin{array}}
\def\ea{\end{array}}
\def\be#1{\begin{equation} \label{#1}}
\def \eeq{\end{equation}}

\def\be{{\beta}}

\def\R{{\mathbb{R}}}

\newtheorem{theorem}{Theorem}[section]
\newtheorem{lemma}[theorem]{Lemma}

\numberwithin{equation}{section}

\numberwithin{equation}{section}

\topmargin - .23 in

\begin{document}

\title{Remark on the scattering theory of the nonlinear Schr\"odinger equation on the cylinders   }
\author{Xing Cheng$^{*}$   and Jiqiang Zheng$^{**}$                       }

\thanks{$^*$ School of Mathematics, Hohai University, Nanjing 210098, Jiangsu,  China. \texttt{ chengx@hhu.edu.cn}}

\thanks{$^{**}$  Institute of Applied Physics and Computational Mathematics and National Key Laboratory of Computational Physics, Beijing 100088, China.
\texttt{zheng\_jiqiang@iapcm.ac.cn}}

\thanks{$^{*}$ X. Cheng has been partially supported by the NSF of Jiangsu (Grant No. BK20221497).    }	
\thanks{$^{**}$  J. Zheng was supported by National key R\&D program of China 2021YFA1002500, NSFC Grant 12271051 and Beijing Natural Science Foundation 1222019}

\begin{abstract}
In this article, we consider the nonlinear Schr\"odinger equation on the cylinder $\mathbb{R}^d\times \mathbb{T}$.
In the long range case, we show there is no linear scattering state of the nonlinear Schr\"odinger equation on $\mathbb{R}^d \times \mathbb{T}$. In the short range case, we show the decay and scattering of
solutions of the nonlinear Schr\"odinger equation on $\mathbb{R}^d \times \mathbb{T}$ for small data.

\bigskip


\bigskip


\end{abstract}
\maketitle
%
%

\section{Introduction}
In this article, we consider the defocusing nonlinear Schr\"odinger equation (NLS) posed on the cylinder $\mathbb{R}^d \times \mathbb{T}$ for $d \ge 1$:
\begin{equation}\label{eq1.2n}
i\partial_t u +  \Delta_{\mathbb{R}^d \times \mathbb{T}}  u =  |u|^{p-1} u,
\end{equation}
where $u(t,y,x) : \mathbb{R} \times \mathbb{R}^d \times \mathbb{T}  \to \mathbb{C}$ is an unknown function, and $1< p < 1 + \frac{4}d$ when $d \ge 1$.

Equation \eqref{eq1.2n} has the following conserved quantities:
\begin{align*}
\text{mass: }    &\quad   \mathcal{M}(u(t))  = \int_{\mathbb{R}_y^d   \times \mathbb{T}_x } |u(t,y,x )|^2\,\mathrm{d}y \mathrm{d}x ,\\
\text{   energy:  }     &  \quad \mathcal{E}(u(t))  = \int_{\mathbb{R}_y^d   \times \mathbb{T}_x }\left( \frac12 |\nabla u(t,y, x )|^2  + \frac1{p+1} |u(t,y, x )|^{p+1} \,\mathrm{d}y\right) \mathrm{d}x .
\end{align*}
The equation \eqref{eq1.2n} in lower dimensions describes wave propagation in nonlinear and dispersive media.
 It also figures in the time-dependent Landau-Ginzburg model of phase transitions.

In the last two decades, there are a lot works on the well-posedness and scattering of defocusing NLS on $\mathbb{R}^d \times \mathbb{T} $.
On one hand, when considering the well-posedness of the Cauchy problem \eqref{eq1.2n}, intuitively, it is determined by the local geometry of the manifold $\mathbb{R}^d \times \mathbb{T} $. The manifold is locally $\mathbb{R}^{d + 1}  $, so the well-posedness is the same as the Euclidean case, that is when $1<  p \le 1 + \frac4{d -1}$ the global well-posedness is expected. Just as the Euclidean case, we say the equation is energy-subcritical when $ 1< p < 1 + \frac4{ d -1}$, $ d \ge 1$ and energy-critical when $p = 1 + \frac4{d -1}$, $ d \ge 2$.
On the other hand, when considering the scattering of \eqref{eq1.2n}, scattering is expected to be determined by the asymptotic volume growth of a ball with radius $r$ in the manifold $\mathbb{R}^d\times \mathbb{T} $ when $r\to \infty$. From the heuristic that linear solutions with frequency $\sim N$ initially localized around the origin will disperse at time t in the ball of radius $\sim N t$, scattering is expected to be determined by the asymptotic volume growth of balls with respect to their radius. Since $\inf_{ z \in \mathbb{R}^d \times \mathbb{T}  } \text{Vol}_{\mathbb{R}^d \times \mathbb{T} } (B(z,r)) \sim r^d, \text{ as } r \to \infty$, the linear solution is expected to decay at a rate $\sim t^{-\frac{d}2}$ and based on the scattering theory on $\mathbb{R}^d$,
the solution of \eqref{eq1.2n} is expected to scatter for $ p \ge 1 + \frac4d$. Moreover,
scattering in the small data case is expected for $ 1 + \frac2d < p < 1 + \frac4d$ when $d \ge 1$. 

 Therefore, regarding heuristic on the well-posedness and scattering, the solution of \eqref{eq1.2n} globally exists and scatters in the range $ 1 + \frac4d \le p \le 1 + \frac4{ d-1}$. For $ 1 + \frac2d < p < 1 + \frac4d$ when $d \ge 1$, scattering is expected as in the Euclidean space case for small data. 
 This heuristic has been justified in \cite{CGYZ,CGZ,TT,HPTV,TV2,Z2}. In this paper, we are interested in the scattering theory of \eqref{eq1.2n} when $1< p < 1 + \frac4d$. It is believed that in the long range case when $1 < p < 1 + \frac2d$, the solutions to \eqref{eq1.2n} do not scatter. And it is also expected that
   in 
  the short range case when $1 + \frac2d < p < 1 + \frac4d$, the solutions to \eqref{eq1.2n} scatter for small data. Therefore, we will give a rigorous proof of these facts.

In the following, we present the main result of this paper.

\begin{theorem}
Let $1 < p < 1 + \frac4d$ and $d \ge 1$. Then, the following statements hold true:
\begin{itemize}
\item If $1< p \leq\min\big\{2,1 + \tfrac{2}d\big\}$, then, the solutions to \eqref{eq1.2n} do not scatter in $L_{y, x }^2( \mathbb{R}^d \times \mathbb{T})$. 
\item 
If $1 + \frac2d < p < 1 + \frac4d$, then the solution to \eqref{eq1.2n} scatters in $H^1$ for small data in $\Sigma$, where
 $\Sigma$ is defined by
 \begin{align*}
\Sigma : = \left\{ f \in H_{y,x}^2: \|f\|_\Sigma := \left\| \langle y  \rangle^2 f \right\|_{L^2_{y,x} } +  \left\| \langle y \rangle \nabla f \right\|_{L^2_{y,x} } + \|f \|_{H^2_{y,x} } < \infty \right\}.
\end{align*}

\end{itemize}

\end{theorem}

The argument of the proof of the non-scattering part is based on the argument of W. A. Strauss \cite{St} for the nonlinear Schr\"odinger equation on the Euclidean space, which relys on the dispersive estimate of $e^{it \Delta_{\mathbb{R}^d \times \mathbb{T}}}$.
The proof of small data scattering is based on the commutator method introduced by H. P. McKean and J. Shatah \cite{MS}. To show small data scattering, we can reduce the scattering to the decay estimates
\begin{align}\label{eq1.4v24}
\|u(t, y , x ) \|_{L_y^\infty H_x^1 ( \mathbb{R}^d \times \mathbb{T} ) } \le C t^{- \frac{d}2}.
\end{align}
Then by introducing a vector field $ \left( |J(t) |^s  u  \right) (t,y, x) $ which is roughly $ \left( t^s  \left( - \Delta_y  \right)^\frac{s}2 e^{-it \Delta_x } u \right) (t, y, x ) $, it suffices to prove
\begin{align}\label{eq1.5v24}
\|u(t,y , x ) \|_{L_y^\infty H_x^1 } \le C t^{- s}  \left\| |J(t) |^s u(t, y, x )  \right\|_{  L_y^2 H_x^1 }^\theta \|u (t) \|_{L_y^2 H_x^1 }^{1- \theta}.
\end{align}
Moreover, $|J(t) |^s u(t) $ satisfies
\begin{align*}
 \left( i \partial_t   + \Delta_y  \right) |J(t) |^s u (t)  =
 |J(t) |^s  e^{-it \Delta_x}  \left( |u(t) |^{p-1} u (t)   \right) . 
\end{align*}
First, we find that the $L_y^{2p} H_x^1 $ decay estimate, roughly
\begin{align}\label{eq1.6v24}
\|u (t,y,x) \|_{L^{2p}_y H_x^1 } \le C t^{- \frac{d}2  \left( 1- \frac1p \right)}  \left\| |J(t) |^s u(t,y, x)   \right\|_{L_t^\infty L_y^2 H_x^1 },
\end{align}
where $s = \frac{d}2  \left( 1- \frac1p \right)$, is enough for scattering. Since $s < 2$ for $p <  1 + \frac4{d}$, each term in the equation of $|J(t) |^s u$ can be estimated easily in our case. Second, to establish \eqref{eq1.6v24}, we transform it to the corresponding estimate of the inverse of $|J(t) |^s$, which can be reduced to the $L^p$ estimate of resolvent.

\subsection{Notation and Preliminaries}
We will use the notation $\mathbb{T} = \mathbb{R}/(2\pi \mathbb{Z})$ is torus. In the following, we will frequently use some space-time norm, we now give the definition of it.

For any time interval $I \subseteq \mathbb{R}$, $u(t,y, x ): I \times \mathbb{R}^d   \times \mathbb{T} \to \mathbb{C}$, define the space-time norm
\begin{align*}
\|u\|_{L_t^q L_y^r L_x^2(I \times \mathbb{R}^d   \times \mathbb{T}) }  & = \left\|\left\|\Big( \int_{\mathbb{T}}   |u(t,y, x )|^2 \,\mathrm{d}x \Big)^\frac12\right\|_{L_y^r(\mathbb{R}^d )}\right\|_{L_t^q (I)},\\
\|u\|_{H^1_{y, x }( \mathbb{R}^d   \times \mathbb{T})}  & =  \left\|\langle \nabla_x \rangle u \right\|_{L_{y, x }^2( \mathbb{R}^d   \times \mathbb{T})} +  \left\|\langle \nabla_y \rangle u \right\|_{L_{y, x }^2( \mathbb{R}^d   \times \mathbb{T})}.
\end{align*}
We will frequently use the partial Fourier transform: For $f(y, x ): \mathbb{R}^d   \times \mathbb{T} \to \mathbb{C}$,
\begin{align*}
 \left( \mathcal{F}_y  f \right) (\xi,x ) = \frac1{(2\pi)^\frac{d}2} \int_{\mathbb{R}^d } e^{-i y \cdot \xi} f(y, x ) \,\mathrm{d}y .
\end{align*}
We denote  $a^\pm$ to be any quantity of the form $a\pm\epsilon$ for any
$\epsilon>0.$

\section{Nonexistence of the linear scattering state in the long range case 
}
In the section, we discuss the long time behaviour of \eqref{eq1.2n} in the long range case for $1< p < 1 + \frac{2}d$ when $d \ge 2$ and $1 < p \le 2$ when $d = 1$. We will show the only asymptotically free solution to \eqref{eq1.2n}  is identically zero.

\begin{theorem}\label{th3.1v11}
If $u$ is a solution of \eqref{eq1.2n}, then for any $h \in L_{y,x}^2(\mathbb{R}^d \times \mathbb{T})$,
\begin{align*}
\left\|u(t)  - e^{ it\Delta_{\mathbb{R}^d \times \mathbb{T}}} h \right\|_{L_{y, x }^2(\mathbb{R}^d \times \mathbb{T})} \nrightarrow 0, \text{ as } t \to \infty.
\end{align*}
\end{theorem}

Before giving the proof, we prove an auxiliary lemma.
\begin{lemma}\label{le3.2v11}
There is a positive constant $c_0$ such that for $t$ large enough, we have
\begin{align*}
t^\frac{d( p-1 ) } 2 \int_{\mathbb{R}^d  \times \mathbb{T}}  \left| \left( e^{ it \Delta_{\mathbb{R}^d  \times \mathbb{T}}} h \right) (y, x ) \right|^{p+1} \, \mathrm{d}y \mathrm{d} x  \ge c_0.
\end{align*}

\end{lemma}
\begin{proof}
By H\"older's inequality, we have
\begin{align*}
& \left( \int_{| y | \le Kt} \int_{\mathbb{T}}  \left| \left( e^{ it\Delta_{\mathbb{R}^d  \times \mathbb{T}}}  h\right) (y, x ) \right|^2 \, \mathrm{d}y  \mathrm{d} x  \right)^\frac{p+1}2 \\
 & \le (Kt)^\frac{d( p-1) } 2  \left\|e^{ it \Delta_{\mathbb{R}^d  \times \mathbb{T}}} h \right\|_{L_y^{p+1} L_x^2( \mathbb{R}^d  \times \mathbb{T})}^{p+1}
\\
& \lesssim  (Kt)^\frac{d( p-1) } 2  \left\|e^{ it \Delta_{\mathbb{R}^d  \times \mathbb{T}}} h \right\|_{L_{y, x } ^{p+1}  ( \mathbb{R}^d  \times \mathbb{T})}^{p+1}  .
\end{align*}
On the other hand, we have
\begin{align*}
& \int_{| y | \le Kt} \int_{\mathbb{T}}  \left| \left( e^{ it \Delta_{\mathbb{R}^d  \times \mathbb{T}}} h  \right) (y, x ) \right|^2 \,\mathrm{d} y  \mathrm{d}x  \\
= & \int_{| y | \le Kt } \int_{\mathbb{T}} \left| \frac1{ (4 \pi i t)^\frac{d}2} \int_{\mathbb{R}^d } e^{  \frac{i \left|y  - \tilde{y} \right|^2 }{4t }} h \left(\tilde{y }, x \right) \, \mathrm{d} \tilde{y} \right|^2 \, \mathrm{d}y  \mathrm{d} x \\
= & \int_{|  y | \le Kt} \int_{\mathbb{T}} \frac1{ (4\pi t)^d } \int\int_{\mathbb{R}^{2d } } e^{ \frac{i \left(\tilde{y}^2 - \tilde{\tilde{y}}^2 \right) + 2 i y \left( \tilde{y} - \tilde{\tilde{y}} \right)}{4t}} \overline{h \left(\tilde{\tilde{y}}, x  \right) } h \left(\tilde{y}, x \right) \,\mathrm{d}\tilde{y} \mathrm{d} \tilde{\tilde{y}} \mathrm{d}x \mathrm{d}y  \\
= & \frac1{ ( 2\pi)^d  } \int_{|\xi | \le \frac{K} 2 } \int_{\mathbb{T}} \int\int_{\mathbb{R}^{2d } } e^{ \frac{i \left(\tilde{y}^2 - \tilde{\tilde{y}}^2 \right)} { 4t} + i \left(\tilde{y} - \tilde{\tilde{y}} \right) \xi} h \left(\tilde{y}, x \right) \overline{h \left(\tilde{\tilde{y}}, x \right)} \,\mathrm{d} \tilde{y} \mathrm{d} \tilde{\tilde{y}} \mathrm{d}x  \mathrm{d} \xi\\
\stackrel{t\to \infty}\to &
 \frac1{ ( 2\pi)^d } \int_{|\xi| \le \frac{K}2} \int_{\mathbb{T}} \left( \mathcal{F}_y  {h}  \right) (\xi,x ) \overline{  \left( \mathcal{F}_y {h}  \right) (\xi, x )}  \,\mathrm{d} \xi \mathrm{d}x
= \frac1{ ( 2\pi)^d } \int_{ |\xi| \le \frac{K} 2} \int_{\mathbb{T}} \left|  \left( \mathcal{F}_y {h}  \right) (\xi, x ) \right|^2 \,\mathrm{d} \xi \mathrm{d}x .
\end{align*}

\end{proof}

Now, we  turn to prove Theorem \ref{th3.1v11}.
\begin{proof}[Proof of Theorem \ref{th3.1v11}.]
Suppose by contradiction that
\begin{align}\label{eq6.2v44}
\left\|u(t) - e^{ it\Delta_{\mathbb{R}^d \times \mathbb{T}}} h \right\|_{L_{y,x }^2} \to 0, \text{ as $t \to \infty$,}
 \end{align}
for some $h \in L_{y,x }^2$.

By the unitary of the operator $e^{ it \Delta_{\mathbb{R}^d \times \mathbb{T}}}$ in $L_{y, x}^2$, we see
\begin{align*}
& \frac{d}{dt} \Im \int_{\mathbb{R}^d \times \mathbb{T}} \left( e^{- it\Delta_{\mathbb{R}^d \times \mathbb{T}}} u \right)(t, y,x ) \cdot \overline{h(y, x )} \, \mathrm{d}y  \mathrm{d} x \\
 = & \Im \int_{\mathbb{R}^d \times \mathbb{T}} e^{- it\Delta_{\mathbb{R}^d \times \mathbb{T}}} \left( i |u|^{p-1} u \right)(t, y, x)  \cdot \overline{h(  y,x) } \, \mathrm{d}y  \mathrm{d}x \\
= & \Im \int_{\mathbb{R}^d \times \mathbb{T}}  \left( i |u|^{p-1} u  \right)(t,y,x) \overline{ \left( e^{ it\Delta_{\mathbb{R}^d \times \mathbb{T}}} h \right)(t,y,x ) } \,\mathrm{d}y  \mathrm{d}x .
\end{align*}
By
\begin{align*}
& \Im \int_{\mathbb{R}^d \times \mathbb{T}}  ( e^{- iT \Delta_{\mathbb{R}^d \times \mathbb{T}}} u)(T, y , x ) \cdot \overline{h(y, x) } \,\mathrm{d} y \mathrm{d} x \\
= & \Im\int_{\mathbb{R}^d \times \mathbb{T}} h(y, x)  \overline{h(y, x) } \,\mathrm{d} y \mathrm{d} x  + \Im \int_{\mathbb{R}^d \times \mathbb{T}}
\left( (e^{- iT \Delta_{\mathbb{R}^d \times \mathbb{T}}} u )(T, y, x)  - h (y, x)  \right) \overline{h(y,x) } \,\mathrm{d}y  \mathrm{d} x  \\
\to &  0, \text{ as } T\to \infty,
\end{align*}
we have
\begin{align}\label{eq2.2newv26}
\int_0^T \Im \int_{\mathbb{R}^d \times \mathbb{T}} i |u|^{p-1} u \cdot \overline{e^{ it \Delta_{\mathbb{R}^d \times \mathbb{T}}} h} \,\mathrm{d}y  \mathrm{d}x  \mathrm{d}t
\end{align}
has a limit as $T\to \infty$. On the other hand, we have
\begin{align*}
& \left| \Im \int_{\mathbb{R}^d \times \mathbb{T}} \left( i |u|^{p-1} u - i  \left|e^{ it\Delta_{\mathbb{R}^d \times \mathbb{T}}} h \right|^{p-1} e^{ it \Delta_{\mathbb{R}^d \times \mathbb{T}}} h\right) \overline{e^{ it \Delta_{\mathbb{R}^d \times \mathbb{T}}} h} \,\mathrm{d}y  \mathrm{d}x  \right|\\
& \lesssim  \left( \|u\|_{L_{y, x }^2} + \|h \|_{L_{y, x }^2} \right)^{p-1} \|h \|_{L_{y, x }^2}^{2-p}
\left\|e^{ it \Delta_{\mathbb{R}^d \times  \mathbb{T}}} h \right\|_{L_{y, x }^\infty}^{p-1}
\left\|u(t) - e^{ it \Delta_{\mathbb{R}^d \times \mathbb{T}}} h \right\|_{L_{y, x }^2},
\end{align*}
where we need the assumption $p\leq2$.
The above inequality together with \eqref{eq6.2v44}, Lemma \ref{le3.2v11}, and
\begin{align*}
\left\|e^{  it \Delta_{\mathbb{R}^d \times \mathbb{T}}} h \right\|_{L_{y, x }^\infty(\mathbb{R}\times \mathbb{T})}
\lesssim  \left\|e^{ it\Delta_{\mathbb{R}^d \times \mathbb{T}}} h \right\|_{L_y^\infty H_x^1}
 \lesssim |t|^{-\frac{d } 2} \|h \|_{L_y^1 H_x^1},
\end{align*}
 yields
\begin{align}\label{eq2.3v17}
\Im \int_{\mathbb{R}^d  \times \mathbb{T}} i |u|^{p-1} u \cdot \overline{e^{  it \Delta_{\mathbb{R}^d \times \mathbb{T}}} h } \, \mathrm{d} y  \mathrm{d} x
\ge \frac{c_0}2 t^{- \frac{d ( p-1) } 2}.
\end{align}
This implies the left side of \eqref{eq2.3v17} is not integrable for $p\leq 1+\tfrac2d$, we have a contradiction to \eqref{eq2.2newv26}.
Thus, we complete the proof of Theorem \ref{th3.1v11}.
\end{proof}

\section{Small data scattering in the short range case}
In the section, we study the long time behaviour of \eqref{eq1.2n} in the short range case for $1 + \frac2d< p < 1 + \frac4d$ when $d \ge 1$.

We consider
\begin{align}\label{eq1.1v24}
\begin{cases}
i \partial_t u + \Delta_{\mathbb{R}^d \times \mathbb{T}}  u  =   |u|^{p-1} u  , \\
u(h, y , x ) = u_0 (y, x  ),
\end{cases}
\end{align}
where $u: [h, \infty) \times \mathbb{R}^d \times \mathbb{T} \to \mathbb{C}$, $h > 0$, $y \in \mathbb{R}^d$, $x \in \mathbb{T}$.

\begin{theorem}\label{th1.1v24}

For $d \ge 1$, $1 + \frac2d < p < 1 + \frac4{d}$, if $\|u_0 \|_{\Sigma}$ is sufficiently small, then the solution to \eqref{eq1.1v24} globally exists. 
Moreover, for any $\gamma < \frac{d}2  \left( 1- \frac1p \right)$, we have the decay estimate
\begin{align}\label{eq1.3v24}
\|u(t, y, x ) \|_{L_y^{2p}H_x^1 (\mathbb{R}^d \times \mathbb{T} ) } \le C t^{- \gamma},
\end{align}
and as a consequence, there exists $u_+ \in H^1$ such that
\begin{align*}
\lim\limits_{t \to \infty}  \left\|u(t) - e^{i t\Delta_{\mathbb{R}^d \times \mathbb{T}}} u_+  \right\|_{H^1} = 0.
\end{align*}

\end{theorem}

First, we recall the basic resolvent estimate, which is extended to the Schr\"odinger operator with inverse square potentials, see \cite{MSZ,MZZ}.

\begin{lemma}[Resolvent estimate] \label{le2.2v24}
The following weighted resolvent estimate holds for $\lambda > 0$:
\begin{align}
 \left\|  \left( \lambda - \Delta_{\mathbb{R}^d}  \right)^{- 1} f \right\|_{L^r(\R^d)} \le C \lambda^{\frac12  \left( d- 2 - d \left( \frac1r + 1 - \frac1q  \right) \right)} \|f\|_{L^q (\R^d)}, \label{eq2.3v24}
\end{align}
where $1 \le q  \le r  \le \infty$.

\end{lemma}

Define the commutator operator: for any $s \in (0, 2) $,
\begin{align*}
|J(t)|^s u(t,y, x)  = M(t)  \left( - t^2 \Delta_{\mathbb{R}^d}  \right)^\frac{s}2 M(- t) e^{-it \Delta_{\mathbb{T}} } u(t,y, x) ,
\end{align*}
where $M(t) = e^{ \frac{i |y |^2}{4t}}$. Moreover, we can see that $|J(t) |^s u(t,y, x)  $ satisfies
\begin{align*}
\left( i \partial_t  + \Delta_{\mathbb{R}^d}  \right)  |J(t) |^s u(t, y, x )   = |J(t) |^s e^{-it \Delta_{\mathbb{T}}  }   \left( |u(t, y, x ) |^{p-1} u(t , y, x )  \right).
\end{align*}
Let $s =  \left( \frac{d}2  \left( 1- \frac1p \right) \right)^+$. By Strichartz estimate, we have
\begin{align}\label{eq3.2v24}
\left\| |J(t) |^s u  \right\|_{L_t^\infty L_y^2 H_x^1 } \le C  \left\| |J(h)|^s u(h) \right\|_{L^2_y H_x^1 } + C  \left\| |J(t) |^s \left( |u|^{p-1} u  \right)  \right\|_{L_t^{q_1'}  L_y^{r_1'}  H_x^1},
\end{align}
where $ \left(q_1 , r_1  \right)$ is a admissible pair, with
\begin{align}
q_1 \in \left(  \frac2{d(p-1) - 2}, \infty \right) .
\label{eq3.4v24}
\end{align}
 In order to apply continuity method, we have to bound $ \left\| |J(t) |^s  \left( |u|^{p-1} u  \right) \right\|_{L_t^{q_1'}  L_y^{r_1'}  H_x^1 }$ by $ \left\| |J(t) |^s u \right\|_{L_t^\infty L_y^2 H_x^1 }$.
 We will first 
  present some properties of the commutator operator, especially ``Sobolev embedding theorem".

\begin{lemma}\label{le3.1v24}

For $u \in H_y^s$, $s =  \left( \frac{d}2  \left( 1- \frac1p \right) \right)^+$, there exists some $0 < \eta < 1$ such that
\begin{align}\label{eq3.7newv29}
\|u \|_{L_y^{2p}} \le C  \left\|  ( - \Delta_{\mathbb{R}^d}  )^\frac{s}2 u  \right\|_{L_y^2} + C  \left\| ( - \Delta_{\mathbb{R}^d} )^\frac{s}2 u  \right\|_{L_y^2}^{1- \eta}  \|u \|_{L_y^2}^{ \eta}.
\end{align}

\end{lemma}

\begin{proof}
We only need to prove
\begin{align}\label{eq3.7v37}
 \left\| ( - \Delta_y  )^{- \frac{s}2} u  \right\|_{L^{2p}} \le C \|u \|_{L^2} + C \|u \|_{L^2}^{1- \eta}   \left\| ( - \Delta_y  )^{- \frac{s}2 } u  \right\|_{L^2}^{ \eta}.
\end{align}

 For $u \in H_y^s $, 
 we have
\begin{align}\label{eq3.8v37}
( - \Delta_y  )^{- \frac{s}2} u  = c(s)^{-1}  \int_0^\infty \lambda^{- \frac{s}2 } ( \lambda - \Delta_y )^{- 1} u \,\mathrm{d} \lambda,
\end{align}
converges strongly in $L^{2p}$, where $c(s) : =  \int_0^\infty t^{- \frac{s}2} ( 1 + \tau)^{-1} \,\mathrm{d} \tau  $.
 In fact, \eqref{eq2.3v24} and Sobolev inequality imply
\begin{align*}
\int_1^\infty  \left\|  \lambda^{- \frac{s}2 }  \left( \lambda- \Delta_y  \right)^{- 1} u  \right\|_{L^{2p}} \,\mathrm{d} \lambda \lesssim 
 \int_1^\infty \lambda^{- \frac{s}2 - 1} \|u \|_{L_y^{2p}} \,\mathrm{d}\lambda \lesssim \|u\|_{H_y^s} ,
\end{align*}
and
\begin{align*}
\int_0^1  \left\| \lambda^{- \frac{s}2}  \left( \lambda - \Delta_y  \right)^{- 1} u  \right\|_{L^{2p}} \,\mathrm{d} \lambda \lesssim  
\|u\|_{L^\alpha} \int_0^1 \lambda^{- \frac{s}2 + \frac12  \left( d-2 - d \left( \frac1{2p} + 1 - \frac1\alpha  \right) \right)} \,\mathrm{d} \lambda \lesssim  \|u \|_{H^s_y  },
\end{align*}
for $\frac1\alpha > \frac1{2p} + \frac{s}d$. Thus, \eqref{eq3.8v37} holds in $L^{2p}_y$.

By \eqref{eq2.3v24}, we obtain
\begin{align}\label{eq3.8v37}
 \left\| \int_1^\infty \lambda^{- \frac{s}2 }  ( \lambda - \Delta_y  )^{- 1} u \,\mathrm{d} \lambda  \right\|_{L^{2p}} \lesssim 
  \|u \|_{L^2} \int_1^\infty \lambda^{- \frac{s}2 + \frac12  \left( d- 2 - d \left( \frac1{2p} + \frac12  \right) \right)} \,\mathrm{d} \lambda \lesssim  \|u \|_{L^2},
\end{align}
where we have used $s =  \left( \frac{d}2  \left( 1- \frac1p  \right) \right)^+$.

By H\"older's inequality, we have
\begin{align}\label{eq3.9v37}
\left\| \int_0^1 \lambda^{- \frac{s}2} ( \lambda - \Delta_y )^{- 1} u \,\mathrm{d}\lambda  \right\|_{L^{2p}}
\le
\left\|\int_0^1 \lambda^{- \frac{s}2 } ( \lambda - \Delta_y )^{-1} u \,\mathrm{d}\lambda  \right\|_{L^\mu}^{ 1- \eta}  \left\|\int_0^1 \lambda^{- \frac{s}2} ( \lambda - \Delta_y )^{- 1} u \,\mathrm{d}\lambda  \right\|_{L^2}^\eta,
\end{align}
where 
 $\frac12 - \frac1\gamma = \frac{s}d$, $\mu > \gamma$, $\frac{1- \eta}\mu + \frac\eta2 = \frac1{2p}$. 
By \eqref{eq2.3v24}, we have 
\begin{align}\label{eq3.10v37}
\left\| \int_0^1 \lambda^{- \frac{s}2 } ( \lambda - \Delta_y )^{- 1} u \,\mathrm{d} \lambda  \right\|_{L^\mu} \lesssim 
 \|u \|_{L^2} \int_0^1 \lambda^{- \frac{s}2 + \frac12
\left( d - 2 - d \left( \frac1\mu + \frac12 \right) \right)} \,\mathrm{d} \lambda \lesssim   \|u \|_{L^2}. 
\end{align}
Again by \eqref{eq2.3v24}, we obtain
\begin{align}\label{eq3.11v37}
\left\|\int_0^1   \lambda^{- \frac{s} 2} ( \lambda - \Delta_y )^{- 1} u \,\mathrm{d} \lambda  \right\|_{L^2}
\lesssim  &
\left\| \int_0^\infty \lambda^{- \frac{s}2 } ( \lambda - \Delta_y )^{- 1} u \,\mathrm{d} \lambda  \right\|_{L^2} + \|u \|_{L^2} \int_1^\infty \lambda^{- \frac{s}2 - 1} \,\mathrm{d} \lambda \\
\lesssim  &   \left\| ( - \Delta_y  )^{- \frac{s}2} u  \right\|_{L^2} +  \|u \|_{L^2}. \notag
\end{align}
Combining \eqref{eq3.8v37}, \eqref{eq3.9v37}, \eqref{eq3.10v37}, and \eqref{eq3.11v37}, 
 we get \eqref{eq3.7v37}. 

\end{proof}

As a direct consequence of Lemma \ref{le3.1v24}, we easily obtain the following lemma.

\begin{lemma}\label{le3.2v24}

Taking $s = \left( \frac{d}2  \left( 1 - \frac1p \right) \right)^+$, 
there exists some $0 < \eta < 1$ such that for $s_0 =  \left( \frac{d}2  \left( 1- \frac1p \right) \right)^-$, we have 
\begin{align*}
\|u (t, y, x ) \|_{L_y^{2p} H_x^1 } \le C t^{- s_0}  \left(  \left\| |J(t) |^{s}  u  \right\|_{  L_y^2 H_x^1 } +   \left\| |J(t) |^{s} u  \right\|_{  L_y^2 H_x^1 }^{1- \eta}
 \|u \|_{L_y^2 H_x^1 }^{ \eta} \right).
\end{align*}

\end{lemma}

Now we turn to the estimate of the nonlinear term.

\begin{lemma}\label{le3.4v24}
There exists $0 < \theta < 1$ such that
\begin{align}\label{eq3.5v24}
& \left\| |J(t) |^s  \left( |u(t) |^{p-1} u(t)   \right)  \right\|_{L_t^{q_1'}  L_y^{r_1'}  H_x^1 }  \\
& \le C  \left(  \left\| |J(t) |^s u (t)   \right\|_{L_t^\infty L_y^2 H_x^1 } + \| u \|_{L_t^\infty L_y^2 H_x^1 }  \right)
\left(  \left\| |J(t) |^s u  \right\|_{L_t^\infty L_y^2 H_x^1 }^\theta +   \left\| |J(t) |^s u  \right\|_{L_t^\infty L_y^2 H_x^1 }^{\theta \eta} \|u \|_{L_t^\infty L_y^2 H_x^1 }^{ ( 1- \eta) \theta}  \right)^{p-1}
\notag \\
& \quad + C(h)
\left(  \left\| |J(t) |^s u  \right\|_{L_t^\infty L_y^2 H_x^1 }^\theta +  \left\| |J(t) |^s u  \right\|_{L_t^\infty L_y^2 H_x^1 }^{\eta \theta} \|u \|_{L_t^\infty L_y^2 H_x^1 }^{( 1 - \eta ) \theta}  \right)^p. \notag
\end{align}

\end{lemma}

\begin{proof}

Let $\tilde{r}_1 = \frac{2r_1}{r_1 + 2}$, 
 we have
\begin{align*}
& \quad \left\| |J(t) |^s  \left( |u|^{p-1} u  \right)  \right\|_{L_t^{q_1'}  L_y^{r_1'}  H_x^1 } \\
& \le
 \left\| ( - \Delta )^\frac{s}2 M(- t)  \left( |u|^{p-1} u \right)  \right\|_{L_t^{q_1'}  L_y^{r_1'} H_x^1 } +  \left\| |u|^{p-1} u  \right\|_{L_t^{q_1'}  L_y^{r_1'} H_x^1  }
\\
& \le  \left\| ( - \Delta)^\frac{s}2  (M(-t ) u )  \right\|_{L_t^\infty L_y^2 H_x^1 }  \left\| |u|^{p-1}  \right\|_{L_t^{q_1'}  L_y^{\tilde{r}_1' } H_x^1 }
+  \left\| \| u \|_{L_y^{ r_1' p} H_x^1 }^p  \right\|_{L_t^{q_1'} } \\
& \le C  \left(  \left\| |J(t) |^s u  \right\|_{L_t^\infty L_y^2 H_x^1 } + \|u \|_{L_t^\infty L_y^2 H_x^1  }  \right)
\left\| |u|^{p-1}  \right\|_{L_t^{q_1'}  L_y^{\tilde{r}_1' } H_x^1 } +  \left\| \|u \|_{L_y^{r_1' p} H_x^1 }^p  \right\|_{L_t^{q_1'} }.
\end{align*}
First, we consider $ \left\| \| u \|_{L_y^{ r_1' p} H_x^1 }^p  \right\|_{L_t^{q_1'} }$. By H\"older's inequality and Lemma \ref{le3.2v24}, we have 
\begin{align*}
\|u \|_{L_y^{ r_1' p} H_x^1 }
\le \|u \|_{L_y^{2p} H_x^1 }^\theta  \|u \|_{L^2_y H_x^1 }^{1- \theta}
\lesssim  t^{- d  \left( \frac12 - \frac1{ r_1' p}  \right) + \epsilon }  \left(  \left\| |J(t) |^s u  \right\|_{  L_y^2 H_x^1 }^\theta +  \left\| |J(t) |^s u  \right\|_{  L_y^2 H_x^1 }^{\theta (1- \eta ) } \| u \|_{L^2_y H_x^1 }^{ \theta \eta }  \right) \|u \|_{L^2_y H_x^1 }^{1 - \theta },
\end{align*}
where $\epsilon $ is sufficiently small. Since $ \left(q_1 ,  r_1 \right)$ is an admissible pair, we have $r_1 = \frac{2d q_1}{dq_1 - 4} $,
 then
\begin{align}\label{eq3.9v34}
& \quad  \left\|  \| u \|_{L_y^{r_1' p} H_x^1 }^p  \right\|_{L_t^{q_1'} } \\
 \lesssim   &
 \bigg( \int_h^\infty t^{\epsilon q_1'  p - d q_1'  p  \Big( \frac12 - \frac{ \frac12 + \frac2{d q_1} 
 }p  \Big)} \,\mathrm{d}t  \bigg)^\frac1{q_1'}
\left(  \left\| |J(t) |^s u  \right\|_{L_t^\infty L_y^2 H_x^1 }^\theta +  \left\| |J(t) |^s u  \right\|_{L_t^\infty L_y^2 H_x^1 }^{(1- \eta ) \theta }
\| u \|_{L_t^\infty L^2_y H_x^1 }^{   \eta  \theta }  \right)^p    \notag \\
\le & C(h)  \left(  \left\| |J(t) |^s u \right\|_{L_t^\infty L_y^2 H_x^1 }^\theta +  \left\| |J(t) |^s u  \right\|_{L_t^\infty L_y^2 H_x^1 }^{(1- \eta)  \theta} \|u \|_{L_t^\infty L^2_y H_x^1 }^{ \eta \theta } \right)^p, \notag
\end{align}
where we have used \eqref{eq3.4v24}. Second, we now turn to the estimate of $ \left\| |u|^{p-1}  \right\|_{L_t^{q_1'} L_y^{\tilde{r}_1'} H_x^1 }$. Similar arguments as the above estimates give
\begin{align}\label{eq3.10v34}
& \quad \left\|  |u|^{p-1}  \right\|_{L_t^{q_1'} L_y^{\tilde{r}_1' } H_x^1  } \\
\lesssim   &
\left( \int_h^\infty t^{\epsilon ( p - 1){q_1'} - d  \left( \frac12 - \frac{ 1}{ r_1' ( p-1 ) }  \right) ( p - 1){q_1'}  } \,\mathrm{d}t  \right)^\frac1{q_1'}
\left(  \left\| |J(t) |^s u  \right\|_{L_t^\infty L_y^2 H_x^1 }^\theta +   \left\| |J(t) |^s u  \right\|_{L_t^\infty L_y^2 H_x^1 }^{(1- \eta ) \theta } \| u \|_{L^\infty_t L^2_y H_x^1 }^{ \eta \theta }  \right)^{p-1} \notag  \\
\le & C(h)  \left(  \left\| |J(t) |^s u  \right\|_{L_t^\infty L_y^2 H_x^1 }^\theta +  \left\| |J(t) |^s u  \right\|_{L_t^\infty L_y^2 H_x^1 }^{(1- \eta )  \theta} \| u \|_{L_t^\infty L^2_y H_x^1 }^{  \eta  \theta }  \right)^{p-1},\notag
\end{align}
where again we have used \eqref{eq3.4v24}. Combining the estimates together, we obtain \eqref{eq3.5v24}.

\end{proof}

\begin{proof}[Proof of Theorem \ref{th1.1v24}]

By \eqref{eq3.2v24} and Lemma \ref{le3.4v24}, we have
\begin{align*}
&  \left\| |J(t) |^s u  \right\|_{L_t^\infty L_y^2 H_x^1 } \\
& \le C  \left(  \left\| |J(t) |^s u  \right\|_{L_t^\infty L_y^2 H_x^1 } + \|u \|_{L_t^\infty L^2_y H_x^1  }  \right)
\left(  \left\| |J(t) |^s u  \right\|_{L_t^\infty L_y^2 H_x^1 }^\theta +
\left\| |J(t) |^s u  \right\|_{L_t^\infty L_y^2 H_x^1 }^{\theta (1- \eta ) } \|u \|_{L_t^\infty L^2_y H_x^1 }^{  \eta  \theta}  \right)^{p-1}\\
& \quad + C(h )
\left(  \left\| |J(t)  |^s u  \right\|_{L_t^\infty L_y^2 H_x^1 }^\theta +  \left\| |J(t) |^s u  \right\|_{L_t^\infty L_y^2 H_x^1 }^{\theta (1 -  \eta ) } \|u \|_{L_t^\infty L^2_y H_x^1 }^{ \eta  \theta }  \right)^p
+ C \|u_0 \|_{\Sigma} + C(h)   \left\| |J(t) |^s u  \right\|_{L_t^\infty L_y^2 H_x^1 }.
\end{align*}
Since $\lim\limits_{h \to \infty } C(h) = 0$, by standard continuity argument, we have for $h$ large enough and $\|u_0 \|_\Sigma$ small enough that
\begin{align}\label{eq3.11v34}
\left\| |J(t) |^s u  \right\|_{L_t^\infty L_y^2 H_x^1 } \le C.
\end{align}
We then have the decay estimate \eqref{eq1.3v24} by Lemma \ref{le3.2v24}.

Finally, we give the proof of scattering as a consequence of \eqref{eq1.3v24}. From Duhamel's principle, it suffices to prove
\begin{align*}
\left\|\int_h^\infty e^{-i s \Delta_{\mathbb{R}^d \times \mathbb{T}} } \left( |u|^{p-1} u \right)(s) \,\mathrm{d}s \right\|_{H^1} \le C.
\end{align*}
By Strichartz estimate, we have
\begin{align*}
\left\| \int_h^\infty e^{-i s \Delta_{\mathbb{R}^d \times \mathbb{T} }}  \left(   |u|^{p-1} u  \right)(s ) \,\mathrm{d}s  \right\|_{H_{y,x}^1}
\lesssim &  \left\|   |u|^{p-1} u  \right\|_{L_t^{q_1'}  L_y^{r_1'} H_x^1 } +  \left\| ( - \Delta_y )^\frac12  \left(   |u|^{p-1} u  \right)  \right\|_{L_t^{q_1'} L_y^{r_1'}  L_x^2 } \\
\lesssim &  \left\|   |u|^{p- 1} u  \right\|_{L_t^{q_1'}  L_y^{r_1'} H_x^1 } +  \|u \|_{L_t^\infty H_y^1 L_x^2 }  \left\| |u|^{p-1}  \right\|_{L_t^{q_1'}  L_y^{\tilde{r}_1' } H_x^1 },
\end{align*}
where $\tilde{r}_1 =  \frac{2r_1}{r_1 + 2} $. 
The argument of the proof of Lemma \ref{le3.4v24} implies
\begin{align*}
 \left\| |u|^{p-1} u  \right\|_{L_t^{q_1'}  L_y^{r_1'} H_x^1 } +    \left\| |u|^{p-1}   \right\|_{L_t^{q_1'}  L_y^{\tilde{r}_1' } H_x^1 } \le C.
\end{align*}
We then define
\begin{align*}
u_+ = e^{- i h \Delta_{\mathbb{R}^d \times \mathbb{T}}  } u_0 - i \int_h^\infty e^{-i \tau \Delta_{\mathbb{R}^d \times \mathbb{T}}  }  \left(   |u|^{p-1} u \right)( \tau ) \,\mathrm{d}\tau,
\end{align*}
and
\begin{align*}
\left\|e^{-i t \Delta_{\mathbb{R}^d \times \mathbb{T}} } u(t) - u_+  \right\|_{H^1} \to 0, \text{ as } t\to \infty,
\end{align*}
which yields the scattering.

\end{proof}


\begin{thebibliography}{99}

\bibitem{BGTV} N. Burq, V. Georgiev, N. Tzvetkov and N. Visciglia, $H^1$ Scattering for Mass-Subcritical NLS
with Short-Range Nonlinearity and Initial
Data in $\Sigma$, Ann. Henri Poincar\'e, 24 (2023), 1355-1376.

\bibitem{CGYZ} X. Cheng, Z. Guo, K. Yang, and L. Zhao, \emph{On scattering for the cubic defocusing nonlinear Schr\"odinger equation on the waveguide $\mathbb{R}^2\times \mathbb{T}$}, Rev. Mat. Iberoam. {\bf 36} (2020), no. 4, 985-1011.

\bibitem{CGZ} X. Cheng, Z. Guo, and Z. Zhao, \emph{On scattering for the defocusing quintic nonlinear Schr\"odinger equation on the two-dimensional cylinder}, SIAM J. Math. Anal. {\bf 52} (2020), no. 5, 4185-4237.


\bibitem{HPTV} Z. Hani, B. Pausader, N. Tzvetkov, and N. Visciglia, \emph{Modified scattering for the cubic Schr\"odinger equation on product spaces and applications}, Forum of Mathematics, PI. (2015), Vol. 3, 1-63.


\bibitem{MS}
H. P. McKean and J. Shatah, \emph{The nonlinear Schr\"odinger equation and the nonlinear heat equation reduction to linear form},
Comm. Pure Appl. Math. {\bf 44} (1991), no. 8-9, 1067-1080.


\bibitem{MSZ}
C. Miao, X. Su, and J. Zheng, \emph{The $ W^{s,p}-$boundedness of stationary wave operators for the Schr\"odinger operator with inverse-square potential},
Trans. Amer. Math. Soc. {\bf 376} (2023), no. 3, 1739-1797.

\bibitem{MZZ}
H. Mizutani, J. Zhang, and J. Zheng, \emph{Uniform resolvent estimates for Schr\"odinger operator with an inverse-square potential},
J. Funct. Anal. {\bf 278} (2020), no. 4, 108350, 29 pp.



\bibitem{St} W. A. Strauss, \emph{Nonlinear scattering theory}, Scattering theory in mathematical physics, edited by J. A. Lavita and J-P. Marchand (Reidel, Dordrecht, Holland, 1974), 53-78.

\bibitem{TT} H. Takaoka and N. Tzvetkov, \emph{On 2D nonlinear Schr\"odinger equations with data on $\mathbb{R}\times \mathbb{T}$}, J. Funct. Anal. {\bf 182}(2001), 427-442.


\bibitem{TV2} N. Tzvetkov and N. Visciglia, \emph{Well-posedness and scattering for NLS on $\mathbb{R}^d\times \mathbb{T}$ in the energy space},  Rev. Mat. Iberoam. {\bf 32} (2016), no. 4, 1163-1188.

\bibitem{Z2} Z. Zhao, \emph{On scattering for the defocusing nonlinear Schr\"odinger equation on waveguide $\mathbb{R}^m \times \mathbb{T}$
(when $m = 2, 3$)}, J. Differential Equations {\bf 275 } (2021), 598-637.
\end{thebibliography}
\end{document}